\documentclass[preprint]{elsarticle}

\usepackage{fourier}
\usepackage{amsthm, amsmath, amsfonts, amssymb}

\usepackage[colorlinks]{hyperref}

\usepackage[ =Diagram]{caption}

\title{Infinite dimensional families  of Calabi--Yau threefolds \\
and moduli of vector bundles}
\author{Edoardo Ballico, Elizabeth Gasparim, {\small and} Bruno Suzuki}
\address{Ballico - Dept. Mathematics, University of Trento, I-38050 Povo, Italy.\newline
Gasparim; Suzuki -  . Matem\'aticas, Univ. Cat\'olica del Norte, Antofagasta, Chile. 
\newline
 ballico@science.unitn.it, etgasparim@gmail.com, obrunosuzuki@gmail.com}

\usepackage{xcolor}

\newtheorem{theorem}{Theorem}[section]
\newtheorem*{theorem*}{Theorem}
\newtheorem{proposition}[theorem]{Proposition}
\newtheorem{lemma}[theorem]{Lemma}
\newtheorem*{lemma*}{Lemma}
\newtheorem{corollary}[theorem]{Corollary}
\newtheorem*{corollary*}{Corollary}

\newtheorem*{conjecture*}{Conjecture}

\theoremstyle{definition}
\newtheorem{definition}[theorem]{Definition}
\newtheorem{question}[theorem]{Problem}
\newtheorem{example}[theorem]{Example}
\newtheorem{notation}[theorem]{Notation}
\theoremstyle{remark}
\newtheorem{remark}[theorem]{Remark}

\newcommand{\ce}{\mathrel{\mathop:}=}

\DeclareMathOperator{\Bun}{Bun}
\DeclareMathOperator{\Ext}{Ext}
\DeclareMathOperator{\End}{End}
\DeclareMathOperator{\EEnd}{\mathcal{E}nd}

\DeclareMathOperator{\Hom}{Hom}
\DeclareMathOperator{\Pic}{Pic}

\DeclareMathOperator{\Tot}{Tot}

\DeclareMathOperator{\HH}{\mathrm{H}}
\DeclareMathOperator{\Spec}{Spec}

\DeclareMathOperator{\coker}{coker}

\begin{document}

\begin{abstract} We study noncompact Calabi--Yau threefolds,
their moduli spaces of vector bundles and   deformation theory. 
We present Calabi--Yau threefolds that have infinitely many  distinct deformations,
constructing them explicitly, 
 and describe the effect that such deformations produce on moduli spaces of vector bundles.
 
\end{abstract}

\maketitle

\tableofcontents

\section*{Introduction}

 Deformation theory of some noncompact surfaces was described in  \cite{BG}, 
together with the effect that such deformations  produce on  moduli spaces of vector bundles. 
In this work we explore the questions of deformations and moduli for 
noncompact Calabi--Yau threefolds. We will see that both deformations and 
moduli present a much richer behaviour in the case of threefolds, producing large families 
of Calabi--Yaus which hold nontrivial moduli spaces of vector bundles (in contrast to 
the case of surfaces, where commutative deformations destroyed such moduli  \cite[Thm.\thinspace 6.14]{BG}). 
We construct families of  threefolds  admitting  infinitely many isomorphism classes of  deformations, 
some of which support nontrivial moduli of vector bundles (Cor.\thinspace\ref{fedmoduli}).
 
 Our motivation to study deformations of noncompact complex varieties comes from mathematical physics, 
 and the persistence of moduli spaces under deformations allows for various interpretations in physics,
where Calabi--Yau   threefolds appear  very often in different contexts. 
For example, deformations of  CY threefolds enter as terms in  the integrals defining
 the action of the theories of Kodaira--Spencer gravity \cite{B}. 
Furthermore, the counting of  BPS states can be computed efficiently on toric CY threefolds, 
as we have  considered  in \cite{GKMR} and \cite{GSTV}. 
 (Note that toric CY threefolds are necessarily noncompact.)  
 Perhaps the most striking occurrences of CY threefolds are in the 
 original formulation of the Mirror Symmetry Conjecture, where 
 a CY threefold $X$ is conjectured to have a mirror CY partner 
 whose Hodge diamond is obtained from the one of $X$ by 
 reflection on the $45^o$ line. These too are to be considered together 
 with their deformation families. 
  
Here we consider
smooth  Calabi--Yau threefolds $W_k$ containing a line $\ell \cong \mathbb{P}^1$.
For the applications we have in mind for future work it will be useful to observe the effect of contracting 
the line to a singularity. 
The existence of a contraction of $\ell$ imposes heavy restrictions on
the normal bundle \cite{ji}, namely $N_{\ell/W}$ must be isomorphic to
one of
\[ \text{(a) \ } \mathcal{O}_{\mathbb{P}^1}(-1) \oplus \mathcal{O}_{\mathbb{P}^1}(-1)
   \text{ ,}\quad \text{(b) \ } \mathcal{O}_{\mathbb{P}^1}(-2) \oplus \mathcal{O}_{\mathbb{P}^1}(0)
   \text{ , \ or \ } \text{(c) \ } \mathcal{O}_{\mathbb{P}^1}(-3) \oplus \mathcal{O}_{\mathbb{P}^1}(+1)
   \text{ .} \]
$W_1$ is the space appearing in the {basic flop}. 
 It is famous in algebraic geometry for being the  simplest example of 
a rational map that is not a blow-up. 
 
We will focus on  the Calabi--Yau threefolds
\[ W_k \ce \Tot\bigl(\mathcal{O}_{\mathbb{P}^1}(-k) \oplus \mathcal{O}_{\mathbb{P}^1}(k-2)\bigr) \text{ for $k\geq 1$.} \]
We will also consider surfaces of the form
$$ Z_k \ce \Tot\bigl(\mathcal O_{\mathbb P^1}(-k)\bigr) .$$

Our main contributions are  descriptions of the deformation theory of such varieties, 
constructing   infinitely many  non-isomorphic deformations  of $W_k$ whenever $k>1$ (see Cor.\thinspace \ref{Wkinf}), and presenting 
the effects that deformations of these threefolds have on their moduli spaces  
of holomorphic vector bundles (see Thm.\thinspace \ref{decrease}). 
We present some nontrivial holomorphic maps between the  deformation spaces (see Thm.\thinspace \ref{maps}).
We also discuss 
deformations from the point of view of affine bundles on $\mathbb P^1$, 
obtaining infinitely many of them for a fixed threefold (see Thm. \thinspace \ref{afbundles}). 

We  make use of  a new definition of commutative deformation, first presented in 
\cite{GKRS}, which is well suited to fit the needs of the noncompact case. 
Our technique to find deformations goes as follows.
Even though there is no well established  deformation theory for noncompact manifolds, 
 we obtain deformations by
working in analogy with Kodaira's theory for the compact case, see \cite{Ko}. Namely, we 
 calculate cohomology with  coefficients in the tangent bundle, and then we proceed 
 to identify which of such directions of infinitesimal deformations  are integrable.
 
 Naturally, when studying total spaces of rank 2 bundles on the complex line, 
 one first ought to  review the corresponding 2 dimensional case, namely, that of  total spaces of rank 1 bundles.
  Hence,  before attacking the case of  threefolds, 
 we first recall the results proved for surfaces.
In the case of  the complex surfaces $Z_k$, with $ k>0$,  the main results about their commutative deformations are:
\begin{enumerate}
\item[S1.] \cite{GKRS} proved that  deformations of the surfaces $Z_k$ can be obtained from the deformations
 of the (compact) Hirzebruch surfaces  $\mathbb{F}_k$, and
 \item[S2.]   \cite{BG} proved that every nontrivial deformation of $Z_k$ is affine.
 The latter  in turn has as immediate corollary:
 \item[S3.]   \cite[Thm.\thinspace 6.14]{BG} showed that  moduli spaces of vector bundles fall to dimension 0 whenever $Z_k$ is deformed 
 classically. 
 \end{enumerate}

In this work we shall prove that the 3 dimensional analogues of results S1, S2, S3 are all false. Indeed, we prove that for $k>1$:

\begin{enumerate}
\item[T1.] Deformation  of the CY threefolds $W_k$ are not obtained from  deformations
 of their compactifications; this follows directly  from Cor. \ref{distinct}. 
 \item[T2.]   $W_k$ has nontrivial deformations which are not affine, see Cor. \ref{notaffine}.
 \item[T3.]   Deformations of $W_k$ can  hold positive dimensional  
  moduli spaces of vector bundles, with some nontrivial deformations preserving  all of the dimensions of moduli
  and others preserving fewer (or none)  of the positive dimensions 
  of the original moduli, see Thm. \ref{decrease}.
 \end{enumerate}

 In this work we consider only classical, i.e. commutative, deformations leaving noncommutative deformations 
for future work. 
 In the case of surfaces \cite{BG2} showed that the effect of noncommutative deformations of the surfaces $Z_k$ 
on moduli of vector bundles is quite the opposite of the statement of property S3, in that, noncommutative 
deformations can have the effect of enlarging the moduli spaces of vector bundles. 
 We expect that  
a similar phenomenon might occur for threefolds, but details remain to be explored later.
 The study of noncommutative deformations requires very different techniques from those considered here. 
 This paper is part of the PhD. Thesis of B. Suzuki at Universidad Cat\'olica del Norte, Chile.

\section{Deformations of noncompact manifolds}

Classical deformation  theory is well understood in the compact case, as explained in the beautiful
  textbook  of Kodaira \cite{Ko}. However, a general theory for the noncompact case is  lacking. 
   In joint work with K\"oppe and Rubilar \cite{GKRS} we studied some features of deformation 
 theory for noncompact Calabi--Yau threefolds,  and we gave a new definition of deformation of complex structure,
 which proved useful. We now recall the basic definitions:
 
 \begin{definition}\label{def}
	A {\it deformation} of a complex manifold $X$ is a holomorphic surjective submersion $\tilde{X}\stackrel{\pi}{\rightarrow}D$, where $D$ is a complex disc centered at $0$ (possibly a vector space, possibly infinite dimensional), satisfying:
	\begin{itemize}
		\item $\pi^{-1}(0)=X$,
		\item $\tilde{X}$ is locally trivial in the $C^{\infty}$  category.
	\end{itemize}
The fibers $X_t \ce \pi^{-1}(t)$ are called \emph{deformations} of $X$.
\end{definition}

\begin{remark}
	Our choice for the dimension of $D$ is $n=h^1(X,TX)$ whenever possible. The case $n=0$ corresponds to the following definition:
\end{remark}

\begin{definition}\label{formal}
We call a manifold $X$ {\it formally rigid} when $\HH^1(X, TX)=0$.
\end{definition}

%We show in \ref{w1} that $W_1$ is formally rigid. 

\begin{definition}\label{rigid}
We call a manifold $X$ {\it rigid} if any deformation 
$\tilde{X}\stackrel{\pi}{\rightarrow}D$ 
is biholomorphic to the trivial bundle $X \times D \to D$.
\end{definition}

 It is too early to say for sure how this  new
 concept will converge to a permanent one;  it ought to  accommodate several 
 refinements and  improvements, for instance
allowing for more general base spaces, and allowing for singularities such as 
orbifold singularities as in \cite{LZ} or toric degenerations such as in \cite{GS}.
 Developing a solid theoretical background is a fundamental  goal of our work in this theme, 
but clearly a large collection of examples needs to be studied first. 
Here we apply Definition \ref{def} to  CY threefolds that are the total space of vector bundles on the projective line,
and construct deformations corresponding to elements of first cohomology with coefficients in the tangent bundle.
 We will see that this method produces large families of  deformations. 

\subsection{The Calabi--Yau threefolds $W_k$}

Let us now describe those Calabi--Yau threefolds which are the total spaces of vector bundles on the projective line
viewed as manifolds.

\begin{definition}\label{WKdef}For $k \geq 1$, we set
\[
W_k = \Tot (\mathcal{O}_{\mathbb{P}^1}(-k) \oplus \mathcal{O}_{\mathbb{P}^1}(k-2)).
\]
The complex manifold structure  can be described by gluing the open sets 
$$U = \mathbb{C}^3_{\{z,u_1,u_2\}}  \quad \mbox{and} \quad  V = \mathbb{C}^3_{\{\xi,v_1,v_2\}}$$ 
by the relation
\begin{equation}\label{canonical}
\boxed{
(\xi, v_1, v_2) = (z^{-1}, z^k u_1, z^{-k+2} u_2)
}
\end{equation}
whenever $z$ and $\xi$ are not equal to 0.
We call (\ref{canonical}) the canonical coordinates for $W_k$.
\end{definition}

We will use $\HH^1(W_k,TW_k)$ to find deformations of $W_k$ even though we do not know if it will provide all deformations
fitting into definition \ref{def}.

\begin{example}
$W_1$ is formally rigid. This was proved in 
\cite{R} by direct calculation showing that
$\HH^1(W_1, TW_1)=0$.
\end{example}
Interestingly, $k=1$ is the only such case, and  for all other values of $k$ the first cohomology groups with tangent coefficients are infinite dimensional. 

\subsection{Infinitely many deformations of $W_2$}

In this section we prove Thm. \ref{W2def}
showing that  the deformation space of $W_2$ contains  
 infinitely many  distinct  isomorphism types of complex threefolds.
 
 We observe that $W_2= Z_2\times \mathbb C$ is the product of a Calabi--Yau surface by affine space, 
 but the surface $Z_2$ has one single nontrivial deformation, which is affine (see \cite{BG}). Therefore 
 any  deformation of the CY threefold $W_2$ that is not affine does not come from deforming the CY surface $Z_2$.

\begin{example}
Computing cohomology with tangent coefficients produces a large and nontrivial family of deformations of $W_2$.
\cite{GKRS} computed an infinite-dimensional family of deformations $\mathcal{W}_2$ of $W_2$
whose elements can be described 
 by gluing $U = \mathbb{C}^3_{\{z, u_1,u_2\}}$ and $V = \mathbb{C}^3_{\{\xi, v_1,v_2\}}$ with the following relations:
\begin{equation} \label{W2family}
 \boxed{ (\xi, v_1, v_2) = \left(z^{-1}, z^2u_1 + \sum_{s \geq 0} t_s z u_2^s, u_2\right) } \text{ .} 
\end{equation}
Deformations are obtained by varying the parameters $t_s$.
\cite{GKRS} proved that this family is nontrivial by showing that it contains both affine and non-affine deformations. 
\end{example}
We now focus on a class of such deformations of $W_2$ indexed by an integer $y$.

\begin{notation} Fix an integer $y \geq 0$.
We  denote by $\mathcal{W}_2(y)$ the deformation of $W_2$ obtained by gluing the charts

\begin{center}
$U=\mathbb{C}^3_{\{z, u_1, u_2\}}$ and $V=\mathbb{C}^3_{\{\xi, v_1, v_2\}}$
\end{center}
with the relation
\begin{align*}
(\xi, v_1, v_2) & = (z^{-1}, z^2u_1 + zu_2^y, u_2)
\end{align*}
 for $z\neq 0$.
\end{notation}

We  show that this family contains infinitely many distinct manifolds, that is, we prove 
that there are infinitely many complex  isomorphism types in \eqref{W2family}.
Hence, we wish to show that for $y_1\neq y_2$ the threefolds 
$W_2(y_1) $ and $W_2(y_2)$ are not isomorphic. Accordingly, we
compute $\HH^1(\mathcal{W}_2(y),T\mathcal{W}_2(y))$.

\begin{lemma} Fix $y \geq 0$. 
The cohomology group $\HH^1(\mathcal{W}_2(y),T\mathcal{W}_2(y))$ is generated as a complex 
vector space by the classes $\sigma_s = \left[\begin{matrix} 0 & z^{-1}u_2^s & 0 \end{matrix}\right]^T$, with $s \geq 0$.
\end{lemma}

\begin{proof}

The transition matrix for the tangent bundle of $\mathcal{W}_2(y)$ is given  by
\[
{\mathbb J}=\left[\begin{matrix}
-z^{-2} & 0 & 0 \\
2zu_1 + u_2^y & z^2 & yzu_2^{y-1} \\
0 & 0 & 1
\end{matrix}\right].
\] 
A 1-cocycle with coefficients in $T\mathcal{W}_2(y)$ may be expressed in $U$ coordinates by
\[
\sigma = \sum_{l =-\infty}^\infty \sum_{i = 0 }^\infty \sum_{s = 0 }^\infty  
\left[\begin{matrix}
 \alpha^1_{lis} \\
 \alpha^2_{lis} \\
 \alpha^3_{lis} \\
\end{matrix}\right]
z^l u_1^i u_2^s .
\] 
We will omit the indices $lis$ from the coefficients $\alpha^1, \alpha^2, \alpha^3$ to simplify notation.
Since monomials having nonnegative powers of $z$ 
are holomorphic on the $U$-chart, we have
\[
\sigma \sim \sum_{l \leq -1 } \sum_{i \geq 0 } \sum_{s \geq 0 }  
\left[\begin{matrix}
 \alpha^1 \\
 \alpha^2 \\
 \alpha^3 \\
\end{matrix}\right]
z^l u_1^i u_2^s,
\] 
where $\sim$ 
denotes cohomological equivalence.
 Changing coordinates:
\begin{align*}
{\mathbb J}\sigma & = \sum_{l \leq -1} \sum_{i \geq 0} \sum_{s \geq 0} 
\left[\begin{matrix}
 -\alpha^1z^{-2} \\
\alpha^1(2zu_1 + u_2^y) + \alpha^2 z^2 + \alpha^3 yzu_2^{y-1} \\
 \alpha^3 \\
\end{matrix}\right]  
z^l u_1^i u_2^s\\
& = \sum_{l \leq -1} \sum_{i \geq 0} \sum_{s \geq 0} 
\left[\begin{matrix}
 -\alpha^1\xi^2 \\
2\alpha^1\xi v_1 + \alpha^2\xi^{-2} + \alpha^3 y \xi v_2^{y-1} \\
 \alpha^3 \\
\end{matrix}\right]
\xi^{-l} (\xi^2v_2-\xi v_2^y)^i v_2^s.
\end{align*}
As the monomials that are holomorphic on $V$ are cohomologous to 0, we obtain that
\begin{align*}
{\mathbb J} \sigma & \sim \sum_{s \geq 0} \alpha^2_{-10s} 
\left[\begin{matrix}
0 \\
\xi^{-1}v_2^s \\
 0 \\
\end{matrix}\right], 
\end{align*}
which proves the lemma.
\end{proof}

Some of the 1-cocycles $\sigma_s$  may still be null-cohomologous. The result depends on the specific deformation $\mathcal{W}_2(y)$ we consider.

\begin{lemma} Fix $y \geq 0$. 
The class of $\sigma_s$ in $\HH^1(\mathcal{W}_2(y),T\mathcal{W}_2(y))$ is cohomologous to 0 if $s \geq y-1$.
\end{lemma}

\begin{proof}
We  divide the proof into three cases.

\textsc{Case 1:}   
Assume $s \geq y \geq 1$. Then $\left[\begin{matrix} -zu_2^{s-y} & 0 & -\frac{2}{y}u_2^{s-y+1} \end{matrix}\right]^T$ is
 holomorphic in $U$ coordinates, so $\sigma_s$ is cohomologous to $\left[\begin{matrix} - zu_2^{s-y} & z^{-1}u_2^s & -\frac{2}{y}u_2^{s-y+1} \end{matrix}\right]^T$, which we now show is null-cohomologous. In fact, changing coordinates we obtain
\begin{align*}
{\mathbb J} \left[\begin{matrix}
 -zu_2^{s-y} \\
 z^{-1}u_2^s \\
 -\frac{2}{y}u_2^{s-y+1} \\
\end{matrix}\right]
& = \left[\begin{matrix}
-z^{-2} & 0 & 0 \\
2zu_1 + u_2^y & z^2 & yzu_2^{y-1} \\
0 & 0 & 1
\end{matrix}\right]
\left[\begin{matrix}
 -zu_2^{s-y} \\
 z^{-1}u_2^s \\
 -\frac{2}{y}u_2^{s-y+1} \\
\end{matrix}\right] \\
& = \left[\begin{matrix}
 z^{-1}u_2^{s-y}\\
 (-2z^2u_1u_2^{s-y} - zu_2^{s}) + zu_2^s - 2zu_2^s \\
 -\frac{2}{y}u_2^{s-y+1} \\
\end{matrix}\right] \\
& = \left[\begin{matrix}
 z^{-1}u_2^{s-y}\\
 -2u_2^{s-y}(z^2u_1+zu_2^y) \\
 -\frac{2}{y}u_2^{s-y+1} \\
\end{matrix}\right] \\
& = \left[\begin{matrix}
 \xi v_2^{s-y}\\
 -2v_1v_2^{s-y} \\
 -\frac{2}{y}v_2^{s-y+1} \\
\end{matrix}\right] \\
& \sim 0,
\end{align*}
since the monomials in the last vector are all holomorphic on $V$.

\textsc{Case 2:} Assume $s=y-1 \geq 0$. Note that
$\left[\begin{matrix} 0 & 0 & -\frac{1}{y} \end{matrix}\right]^T$ is holomorphic on $U$ coordinates. 
Then $\sigma_s$ is cohomologous to $\left[\begin{matrix} 0 & z^{-1}u_2^s& -\frac{1}{y} \end{matrix}\right]^T$, which 
we now prove is null-cohomologous. In fact, changing coordinates we have
\begin{align*}
\mathbb{J} \left[\begin{matrix}
 0 \\
 z^{-1}u_2^s \\
 -\frac{1}{y} \\
\end{matrix}\right] 
& = \left[\begin{matrix}
-z^{-2} & 0 & 0 \\
2zu_1 + u_2^y & z^2 & yzu_2^{y-1} \\
0 & 0 & 1
\end{matrix}\right]
\left[\begin{matrix}
 0 \\
 z^{-1}u_2^s \\
 -\frac{1}{y} \\
\end{matrix}\right] \\
 & = \left[\begin{matrix}
 0 \\
 0 \\
 -\frac{1}{y}u_2 
\end{matrix}\right] \\
& =  \left[\begin{matrix}
 0 \\
 0 \\
 -\frac{1}{y}v_2 
\end{matrix}\right] \\
& \sim 0,
\end{align*}
since the  the last vector is holomorphic on $V$.

\textsc{Case 3:} Assume $y=0$. Note that $[\begin{matrix}  zu_2^s & 0 & 0 \end{matrix}]^T$
 is holomorphic on $U$ coordinates.

In this case $\sigma_s$ is cohomologous to $[\begin{matrix}  zu_2^s & z^{-1}u_2^s & 0 \end{matrix}]^T$. 
Then changing coordinates:
\begin{align*}
\mathbb{J} \left[\begin{matrix}
 zu_2^s \\
 z^{-1}u_2^s \\
 0 \\
\end{matrix}\right] 
& = \left[\begin{matrix}
-z^{-2} & 0 & 0 \\
2zu_1 + u_2^y & z^2 & yzu_2^{y-1} \\
0 & 0 & 1
\end{matrix}\right]
\left[\begin{matrix}
 zu_2^s \\
 z^{-1}u_2^s \\
 0 \\
\end{matrix}\right] \\
& = \left[\begin{matrix}
 -z^{-1}u_2^s \\
 2u_2^s(z^2u_1+1) \\
 0 \\
\end{matrix}\right] \\
& = \left[\begin{matrix}
 - \xi v_2^s \\
 2 v_1 v_2^s \\
 0 \\
\end{matrix}\right] \\
& \sim 0,
\end{align*}
since  the last vector is holomorphic on $V$.

\end{proof}

\begin{corollary} For every $s \geq 0$ we have the following bounds for the dimensions of the cohomology groups of  deformations with coefficients on the tangent bundle:
\begin{itemize}
\item $h^1(W_2(0),TW_2(0)) = h^1(W_2(1),TW_2(1)) = 0$.
\item $h^1(W_2(y),TW_2(y)) \leq y-1$ for $y \geq 2$.
\end{itemize}

\end{corollary}

We now show that these bounds are sharp.

\begin{theorem}\label{h1s} For $y \geq 2$ we have
$h^1(W_2(y),TW_2(y)) = y-1$ .
\end{theorem}

\begin{proof} Fix $y \geq 2$. 
First we show that the cocycles $[\sigma_s]= \left[\begin{matrix} 0 & z^{-1}u_2^s & 0 \end{matrix}\right]^T$   are nontrivial for 
$ s=0, \ldots, y-2 $.

Suppose $\sigma_s$ is a coboundary. Then there exist functions $\alpha$ holomorphic on $U$ 
and $\beta$ holomorphic on $V$ such that
\[
\sigma_s = \alpha + T^{-1}\beta.
\]

Thus, omitting indices $lis$ from the coefficients of $\alpha$ and $\beta$ we have an expression of the form:
\begin{align*}
\left[\begin{matrix}
 0 \\
 z^{-1}u_2^s \\
 0 \\
\end{matrix}\right] 
& = \sum_{l,i,s \geq 0} 
 \left[\begin{matrix}
 \alpha^1 \\
\alpha^2 \\
\alpha^3 \\
\end{matrix}\right] 
z^l u_1^i u_2^s
+
\left[\begin{matrix}
-\xi^{-2} & 0 & 0 \\
2\xi v_1 - v_2^y & \xi^2 & -y\xi v_2^{y-1} \\
0 & 0 & 1
\end{matrix}\right]
\left[\begin{matrix}
\beta^1 \\
\beta^2 \\
\beta^3 \\
\end{matrix}\right] 
\xi^l v_1^i v_2^s\\
 & = \sum_{l,i,s \geq 0} 
  \left[\begin{matrix}
\alpha^1 \\
\alpha^2 \\
\alpha^3 \\
\end{matrix}\right] 
z^l u_1^i u_2^s
+
\left[\begin{matrix}
 -\beta^1 \xi^{-2} \\
\beta^1(2\xi v_1 - v_2^y) +\beta^2\xi^2 -\beta^3 j \xi v_2^{y-1} \\
\beta^3 \\
\end{matrix}\right] 
\xi^l v_1^i v_2^s \\
 & = \sum_{l,i,s \geq 0} 
 \left[\begin{matrix}
\alpha^1 \\
\alpha^2 \\
\alpha^3 \\
\end{matrix}\right] 
z^l u_1^i u_2^s
+
\left[\begin{matrix}
 -\beta^1 z^2 \\
\beta^1(2z u_1 + u_2^y) +\beta^2z^{-2} -\beta^3 y z^{-1} u_2^{y-1} \\
\beta^3 \\
\end{matrix}\right] 
z^{-l} (z^2u_1+zu_2^y)^i u_2^s \\
 & = \sum_{l,i,s \geq 0} 
  \left[\begin{matrix}
\alpha^1 \\
\alpha^2 \\
\alpha^3 \\
\end{matrix}\right] 
z^l u_1^i u_2^s 
+
\left[\begin{matrix}
 -\beta^1 z^2 \\
\beta^1(2z u_1 + u_2^y) +\beta^2z^{-2} -\beta^3 y z^{-1} u_2^{y-1} \\
\beta^3 \\
\end{matrix}\right] 
z^{-l+i} (zu_1+u_2^y)^i u_2^s.
\end{align*}
But on the right-hand side of the equation the monomials of the form $z^{-1}u_1^0u_2^s$ appear only for $s \geq y-1$. So it is impossible to solve for $\alpha$ and $\beta$.
Hence, we have shown that each $\sigma_s$ is nonzero in cohomology. It remains to show that they are linearly independent.

Assume otherwise, that there is a linear dependence among the cohomology classes of the $\sigma_s$. Such a relation would then 
be given by a polynomial on these classes whose class is  a coboundary. 
 Let $p$ be any  polynomial on $u_2$  that   has degree at most $y-2$ 
 and let $ \sigma_p = [\begin{matrix} 0 & z^{-1}p(u_2) & 0 \end{matrix}]^T$. We wish to show that $\sigma_p$ is not a coboundary.
 Suppose there exist functions $\alpha$ holomorphic on $U$ and $\beta$ holomorphic on $V$ such that
\[
\sigma_p = \alpha + T^{-1}\beta.
\]
Analogously to the first part of the proof, in coordinates we would have
\begin{align*}
\left[\begin{matrix}
 0 \\
 z^{-1}p(u_2) \\
 0 \\
\end{matrix}\right] 
 & = \sum_{l,i,s \geq 0} 
  \left[\begin{matrix}
\alpha^1 \\
\alpha^2 \\
\alpha^3 \\
\end{matrix}\right] 
z^l u_1^i u_2^s 
+
\left[\begin{matrix}
 -\beta^1 z^2 \\
\beta^1(2z u_1 + u_2^y) +\beta^2z^{-2} -\beta^3 y z^{-1} u_2^{y-1} \\
\beta^3 \\
\end{matrix}\right] 
z^{-l+i} (zu_1+u_2^y)^i u_2^s.
\end{align*}
But on the right-hand side of the equation the monomials of the form $z^{-1}u_1^0u_2^s$ appear only for $s \geq y-1$. So it is impossible to solve for $\alpha$ and $\beta$. It follows that the sections $\sigma_0, \ldots, \sigma_{y-2}$ are pairwise non cohomologous.
\end{proof}

We have thus proved the following result:

\begin{theorem}\label{W2def}
Let $y_1, y_2 \geq 2$. Then $\mathcal{W}_2(y_1)$ is isomorphic to $\mathcal{W}_2(y_2)$ if and only if $y_1 = y_2$. Hence the family \eqref{W2family} contains infinitely many distinct isomorphism classes of complex manifolds.
\end{theorem}

Observe  that Theorem \ref{h1s} also implies:

\begin{corollary}\label{notaffine} For $y \geq 2$ the threefolds
$W_2(y)$ are not affine.
\end{corollary}

\subsection{Deformations of $W_3$  as affine bundles}

We use  $\HH^1(W_3,TW_3)$  to parametrise formal infinitesimal deformations of $W_3$.
In this section we will regard these deformations as affine line bundles on the surface $Z_{-1}$,
as defined in \ref{afbundle}. Even though in a certain sense the results of this section 
on affine bundles are somewhat weaker then those of section \ref{general} 
which study their total spaces, the former do not follow from the latter, 
and they are of  independent interest, so we have decided to present 
both points of view. The reader interested only in deformations of manifolds may skip this subsection.
Recall from definition \eqref{canonical} that
 $W_3$ can be covered by $U=\{(z,u_1,u_2)\}$ and $V=\{(\xi, v_1, v_2)\}$, 
with $U \cap V = \mathbb C-\{0\} \times \mathbb C^2$ and transition function given by:
\begin{equation}\label{W3}
\boxed{(\xi, v_1, v_2) = (z^{-1}, z^3u_1, z^{-1}u_2)}
\end{equation}
An infinite dimensional family parametrising deformations  of $W_3$ is given by:
 
\begin{lemma}\cite[Thm.\thinspace 20]{GKMR}\label{par3}
There is a semiuniversal 
 deformation space $\mathcal{W}$  for $W_3$  parametrised by cocycles of the form
$$
\left[
\begin{array}{cccr}
a_{lis} \\
b_{lis} \\
c_{lis}
\end{array}
\right] z^l u_1^i u_2^s
\qquad  3i-3-l-s<0.
$$
\end{lemma}

But, it is a priori still possible that the family produces only finitely many isomorphism types.
We first show that the family given in lemma \ref{par3} does indeed produce infinitely many 
integrable directions. 

\begin{lemma} \label{invertible1}
A cocycle of the form 
\[
\left[\begin{matrix}
0 \\
z^lu_1^iu_2^s \\
0
\end{matrix} \right]
\]
defines a deformation of $W_3$ if and only if $i=0$.
\end{lemma}
\begin{proof}
The deformation is given by
\begin{align*}
\left[ \begin{matrix}
\xi \\ v_1 \\ v_2
\end{matrix} \right]
& =
\left[ \begin{matrix}
z^{-2} & 0 & 0\\
0 & z^3 & 0 \\
0 & 0 & z^{-1}
\end{matrix} \right]
\left(
\left[ \begin{matrix}
z \\ u_1 \\ u_2
\end{matrix} \right]
+
\left[ \begin{matrix}
0 \\ z^l u_1^i u_2^s \\ 0
\end{matrix} \right]
\right) \\
& = 
\left[ \begin{matrix}
z^{-1} \\ z^3u_1 + z^{l+3} u_1^i u_2^s \\ z^{-1}u_2
\end{matrix} \right].
\end{align*}
This rule defines a change of coordinates (an invertible function) if and only if $i=0$.

Indeed, we have $z = \xi^{-1}$ and $u_2 = \xi^{-1} v_2$.

Then
\[
v_1 = \xi^{-3}u_1 + \xi^{-l-s-3}u_1^iv_2^s,
\]
which does not admit a unique solution for $u_1$ if $i \neq 0$.
\end{proof}

\begin{lemma}
A cocycle of the form 
\[
\left[\begin{matrix}
0 \\
0 \\
z^lu_1^iu_2^s
\end{matrix} \right]
\]
defines a deformation of $W_3$ if and only if  $s=0$.
\end{lemma}
\begin{proof}
The deformation is given by
\begin{align*}
\left[ \begin{matrix}
\xi \\ v_1 \\ v_2
\end{matrix} \right]
& =
\left[ \begin{matrix}
z^{-2} & 0 & 0\\
0 & z^3 & 0 \\
0 & 0 & z^{-1}
\end{matrix} \right]
\left(
\left[ \begin{matrix}
z \\ u_1 \\ u_2
\end{matrix} \right]
+
\left[ \begin{matrix}
0 \\ 0 \\ z^l u_1^i u_2^s
\end{matrix} \right]
\right) \\
& = 
\left[ \begin{matrix}
z^{-1} \\ z^3u_1 \\ z^{-1}u_2 + z^{l-1}u_1^iu_2^s
\end{matrix} \right].
\end{align*}
This rule defines a change of coordinates  (an invertible  function) if and only if $s=0$. The proof is analogous to the proof of Lemma \ref{invertible1}.
\end{proof}

\begin{lemma}\label{cocyclesw3}
The cocycles 
\[
\left[
\begin{matrix}
0 \\ z^lu_2^s \\ 0
\end{matrix}
\right]
\]
are nonzero in $\HH^1(W_3,TW_3)$ for $l=-1, -2$ and $s \geq 0$, and pairwise distinct. 
\end{lemma}

\begin{proof}
A general 1-coboundary $\tau$ is given by
\begin{align*}
\tau & = \sum_{l \geq 0} \sum_{i \geq 0} \sum_{s \geq 0} 
\left[ \begin{matrix}
\alpha^1 \\ \alpha^2 \\ \alpha^3
\end{matrix} \right]
z^l u_1^i u_2^s
+
\left[ \begin{matrix}
-\xi^{-2} & 0 & 0 \\
3\xi^2v_1 & \xi^3 & 0 \\
-\xi^{-2}v_2 & 0 & \xi{-1}
\end{matrix} \right]
\left[
\begin{matrix}
\beta^1 \\ \beta^2 \\ \beta^3
\end{matrix}
\right]
\xi^l v_1^i v_2^s \\
 & = \sum_{l \geq 0} \sum_{i \geq 0} \sum_{s \geq 0} \left[ \begin{matrix}
\alpha^1 \\ \alpha^2 \\ \alpha^3
\end{matrix} \right]
z^l u_1^i u_2^s
+
\left[
\begin{matrix}
-\beta^1 z^2 \\ 3\beta^1 z^3u_1 +  \beta^2 z^{-3} \\ -\beta^1 z u_2 - \beta^3 z
\end{matrix}
\right]
z^{-l+3i-s} u_1^i u_2^s.
\end{align*}
So we see that on the second entries of the matrices the monomials $z^{-1}u_2^s$ and $z^{-2}u_2^s$ do not appear.
\end{proof}

\begin{proposition} \label{defW3}
The following infinite-dimensional family of deformations of $W_3$ is obtained by integrating cocycles in  $\mathcal W$:
$$ 
(\xi, v_1,v_2)=
\left(
z^{-1},
z^3 u_1 + \sum_{s \geq 0} \left(t_s z^2+  t'_s z\right) u_2^s ,
z^{-1}u_2 \right) .
$$
\end{proposition}

\begin{proof}
This family is obtained from the cocycles of the form

\[
\left[
\begin{matrix}
0 \\ z^{-2}u_2^s \\ 0
\end{matrix}
\right]
\textnormal{ and }
\left[
\begin{matrix}
0 \\ z^{-1}u_2^s \\ 0
\end{matrix}
\right].
\]

\end{proof}

%%%
We now present a  result about the deformations of $W_3$ given in Proposition  \ref{defW3}
regarded as affine rank 1 bundles over a surface. 
Recall that for each integer  $k$ the surface $Z_k$ can be described in charts by gluing two copies of 
$\mathbb{C}^2$ with coordinates $(z,u)$ and $(\xi, v)$ and
with change of coordinates on $\mathbb C^*\times \mathbb C$ given by
$
(\xi, v) = (z^{-1}, z^k u).
$
In particular the surface $ Z_{(-1)}$ can be  described by gluing $
(\xi, v) = (z^{-1}, z^{-1} u).
$
We now consider deformations  $\mathcal{W}_3(j)$ which may also be regarded as the total space of an  affine line bundle over 
the surface $ Z_{(-1)}$.

\begin{notation}\label{afbundle}
 Fix $j$ a positive integer. Then by $\mathcal{W}_3(j)$ we mean the deformation of $W_3$ given by the transition
function
\[
(\xi, v_1,v_2)=
\left(
z^{-1},
z^3 u_1 + z^2u_2^j ,
z^{-1}u_2 \right).
\]
The threefold $\mathcal{W}_3(j)$ can also be regarded as the total space of the {\it rank 1 affine bundle} 
$$E(j) \stackrel{\pi}{\to} \mathbb Z_{-1},$$ where $\pi$ is the projection on the first and third coordinates and the transition function for $E(j)$  is given in canonical coordinates by $u_1 \mapsto z^3u_1 + z^2u_2^j$.

\end{notation}

To identify the structure of affine bundle, we just write $(z,u_2)$ and $(\xi, v_2)$ as the coordinates of $ Z_{(-1)}$. 

\begin{theorem}\label{afbundles}
If $j_1\neq j_2$ then $E(j_1)$ and  $E(j_2)$ are not isomorphic as affine bundles. 
\end{theorem}

\begin{proof} We may assume, without loss of generality, that $j_1 < j_2$.
Suppose the bundles were isomorphic, and let $T: E(j_1) \to E(j_2)$ be an affine bundle isomorphism.
Then $T$ is an affine transformation on each fiber, i.e., there exist holomorphic functions 
\[
\begin{array}{rcl}
A^U\colon \mathbb{C}^2_{z,u_2} & \to & \textrm{GL}(1) =  \mathbb{C} - \{0\} \\
A^V \colon \mathbb{C}^2_{\xi, v_2} & \to & \textrm{GL}(1) = \mathbb{C} - \{0\} \\
b^U \colon \mathbb{C}^2_{z,u_2} & \to & \mathbb{C}\\
b^V \colon \mathbb{C}^2_{\xi,v_2} & \to & \mathbb{C}
\end{array}
\]
such that
\begin{align}
\label{affine1}
T_{z,u_2}(u_1) & = A^U_{z,u_2}(u_1) + b^U_{z,u_2} \\
\label{affine2}T_{\xi, v_2} (v_1) & = A^V_{\xi, v_2}(v_1) + b^V_{\xi, v_2}.
\end{align}
Let $\phi_U$ and $\varphi_U$ denote the transition functions  
from $(z, u_1, u_2)$ coordinates to $(\xi, v_1, v_2)$ coordinates of $E(j_1)$ and $E(j_2)$, respectively.
The relations \eqref{affine1} and \eqref{affine2} should agree on the intersection of $U$ and $V$, i.e., Diagram \ref{diag1} must commute. Then
\begin{align*}
T \circ \phi_U (z, u_1, u_2) & = \varphi_U \circ T (z, u_1, u_2) \\
T (z^{-1}, z^3u_1 + z^2u_2^{j_1}, z^{-1}u_2) & = \varphi_U (z, A^U_{z, u_2}(u_1) + b^U_{z,u_2}, u_2) \\
(z^{-1}, A^V_{z^{-1},u_2}(z^3u_1 + z^2u_2^{j_1}) + b^V_{z^{-1}, u_2} ,z^{-1}u_2 ) & = (z^{-1}, z^3(A^U_{z, u_2}(u_1) + b^U_{z,u_2}) + z^2u_2^{j_2} ,z^{-1}u_2).
\end{align*}

By comparing the second coordinates on both sides of the equations we have:
\begin{align*}
A^V_{z^{-1},u_2}\left(z^3u_1 + z^2u_2^{j_1}\right) + b^V_{z^{-1},u_2} 
& = z^3\left(A^U_{z, u_2}(u_1) + b^U_{z,u_2}\right) + z^2u_2^{j_2}\\
A^V_{z^{-1},u_2}z^3u_1 + A^V_{z^{-1},u_2}z^2u_2^{j_1} + b^V_{z^{-1},u_2} 
& = A^U_{z,u_2}z^3u_1 + b^U_{z,u_2}z^3 + z^2u_2^{j_2}.
\end{align*}
By comparing the linear and affine parts we get:
\begin{align}
\label{compare1}
 A^V_{z^{-1},u_2}z^3  
 & = A^U_{z,u_2}z^3  \\
 \label{compare2}
 A^V_{z^{-1},u_2}z^2u_2^{j_1} + b^V_{z^{-1},u_2} 
 & =  b^U_{z,u_2}z^3 + z^2u_2^{j_2} .
\end{align}
The only solution is $b^U_{z,u_2} = b^V_{z^{-1},u_2} = 0$ and $A^V_{z^{-1},u_2}  
 = A^U_{z,u_2} = u_2^{j_2-j_1}$, which is not possible because then $A^U$ and $A^V$ would vanish on $u_2=0$ and $v_2=0$ respectively and by definition they must be everywhere non-zero.

We conclude that the bundles $E(j_1)$ and $E(j_2)$ are not isomorphic.
\begin{figure} 
\[
\begin{array}{ccccc}
E(j_1) \supset U \supset & U \cap V & \xrightarrow[]{T_U(\cdot) = A^U(\cdot) + b^U} & U' \cap V' & \subset U' \subset E(j_2) \\
&\left. \phi_U \right\downarrow & & \left\downarrow \varphi_U \right. & \\
E(j_1) \supset V \supset & U \cap V &
\xrightarrow[T_V(\cdot) = A^V(\cdot) + b^V]{}
& U' \cap V' & \subset V' \subset E(j_2)  
\end{array}
\]
\caption{Diagram illustrating the commutativity of the transition functions of $E(j_1)$ and $E(j_2)$ with the restrictions of $T$ to the intersections $U \cap V \subset U$ and $U \cap V \subset V$.}
\label{diag1}
\end{figure}
\end{proof}

\subsection{Holomorphic maps between the $W_k$'s.}

We describe holomorphic maps between $W_2$ and $W_3$ and between their deformation families. 
Existence of such holomorphic maps is not at all  a priori guaranteed.
 
\begin{lemma} \label{phi} The map $\varphi\colon W_2 \rightarrow W_3$ defined by 

$$\varphi\vert_U (z,u_1,u_2) =  (z, zu_1^2,u_2) $$
$$\varphi\vert_V (\xi, v_1,v_2) =  (\xi, v_1^2, \xi v_2)$$
is holomorphic. 
\end{lemma}

\begin{proof}

We prove that $T_3 \circ \varphi_U=\varphi_V \circ T_2$:

\begin{align*}
T_3 \circ \varphi_U (z, u_1, u_2) & = 
T_3 ( z, zu_1^2, u_2 ) \\
& = (z^{-1}, z^4u_1^2, z^{-1}u_2) \\
& = \varphi_V (z^{-1},z^2u_1 ,u_2 ) \\
& = \varphi_V \circ T_2(z, u_1, u_2) .
\end{align*}

In a diagram:

\[
\begin{array}{ccc}
(z, u_1, u_2) & \stackrel{\varphi_U}{\longmapsto} & (z, zu_1^2u_2) \\
 T_2 \downarrow & & \downarrow T_3\\
(z^{-1}, z^2u_1, u_2) & \stackrel{\varphi_V}{\longmapsto} & (z^{-1}, z^4u_1^2, z^{-1}u_2) \\
(\xi, v_1,v_2) & \stackrel{\varphi_V}{\longmapsto} & (\xi, v_1^2, \xi v_2)
\end{array}
\]

\end{proof}

\begin{lemma} \label{phi3} The map $\psi\colon W_3 \rightarrow W_2$ defined by 

$$\psi\vert_U (z,u_1,u_2) =  (z, u_1, z^2u_1u_2) $$
$$\psi\vert_V (\xi, v_1,v_2) =  (\xi, \xi v_1, v_1 v_2)$$
is holomorphic. 
\end{lemma}

\begin{proof}

We prove that $T_2 \circ \psi_U = \psi_V \circ T_3$:

\begin{align*}
T_2 \circ \psi_U (z, u_1, u_2) & = 
T_2 (z, u_1, z^2u_1u_2) \\
& = (z^{-1}, z^2u_1, z^2u_1u_2) \\
& = \psi_V (z^{-1},z^3u_1 , z^{-1}u_2 ) \\
& = \psi_V \circ T_3 (z, u_1, u_2) .
\end{align*}

In a diagram:
\[
\begin{array}{ccc}
(z, u_1, u_2) & \stackrel{\psi_U}{\longmapsto} & (z, zu_1^2u_2) \\
 T_3 \downarrow & & \downarrow T_2\\
(z^{-1}, z^2u_1, u_2) & \stackrel{\psi_V}{\longmapsto} & (z^{-1}, z^4u_1^2, z^{-1}u_2) \\
(\xi,v_1,v_2) & \stackrel{\psi_V}{\longmapsto} & (\xi, v_1^2, \xi v_2)
\end{array}
\]\end{proof}

Extending the map to deformations  and repeating a proof similar to the one of Lemma \ref{phi3}  we obtain a 
holomorphic map between deformation spaces:

\begin{proposition} \label{maps} The map $\overline{\varphi} \colon \mathcal{W}_3 \to \mathcal{W}_2$ 
\begin{align*}
\overline{\varphi}\vert_U (z,u_1,u_2) & =  \left( z, u_1,z^2u_1u_2 + z\sum_{s \geq 0}t_su_2^{s+1} \right) \\
\overline{\varphi}\vert_V (\xi, v_1,v_2) & =  (\xi, \xi v_1, v_1 v_2)
\end{align*}
is holomorphic.
\end{proposition}

\subsection{Infinitely many deformations of $W_k$}\label{general}
In further generality we might consider 
$$
W_{k_1,k_2} = \Tot (\mathcal{O}_{\mathbb{P}^1}(-k_1) \oplus \mathcal{O}_{\mathbb{P}^1}(-k_2)),
\quad  \mbox{with} \,\,
k_1 \geq k_2.$$
Using the same methods from the previous sections we deduce  the following  results:
\begin{itemize}
\item If $0 \geq k_1 \geq  k_2 $ then the threefold is formally rigid, i.e., $\HH^1(W_{k_1,k_2}, TW_{k_1,k_2})=0$.
\item If $k_1 \geq k_2 > 0$ then the threefold has a finite dimensional deformation space.

\item If  $k_1 > k_2=0$, then the threefold has an infinite dimensional deformation space.

\item If $k_1 > 0 > k_2$, then  the threefold has an infinite dimensional deformation space.

\end{itemize}

 Using Theorem \ref{W2def} we will now also construct infinitely many deformations for all cases $k > 2$, thus proving:
 
\begin{theorem} \label{distinct}There are infinitely many distinct deformations of \, $W_k$ for $k>1$.
\end{theorem}

The idea of the proof comes from observing deformations of vector bundles of $\mathbb P^1$ as illustrated in the following example.

\begin{example}
Consider the family of rank 2 vector bundles on $\mathbb P^1$ parametrised by $t$, described by transition functions 
in canonical coordinates as
$$
\left[
\begin{matrix}
z^3 & tz^2 \\ 0 & z^{-1}
\end{matrix}
\right].$$
Taking total spaces of these deformations, when $t=0$ we obtain $W_3$ and when $t=1$ we obtain $W_2$. 
Indeed, making a change of coordinates
$$\left[
\begin{matrix}
z^3 & z^2 \\ 0 & z^{-1}
\end{matrix}
\right] \left[
\begin{matrix}
1 & 0 \\ -z & 1
\end{matrix}
\right]=\left[
\begin{matrix}
0 & z^2 \\ -1 & z^{-1}
\end{matrix}
\right]\sim \left[
\begin{matrix}
z^2 & 0 \\ z^{-1} & -1
\end{matrix}
\right].$$
Changing of coordinates again
$$\left[
\begin{matrix}
1 & 0 \\ -z^{-3} & 1
\end{matrix}
\right] \left[
\begin{matrix}
z^2 & 0 \\ z^{-1} & -1
\end{matrix}
\right]=\left[
\begin{matrix}
z^2 & 0 \\ 0 & -1
\end{matrix}
\right].
 $$
 This example is just a concrete way to write in coordinates a family of extensions
 of the form 
 $$0 \rightarrow \mathcal O_{\mathbb P^1} (-3) \rightarrow E_t \rightarrow \mathcal O_{\mathbb P^1}(1) \rightarrow 0$$ such that 
 $$W_t = \Tot (E_t).$$
 \end{example}
 
This implies that deformations of $W_2$ induce deformations on $W_3$. 
Similarly, we can use this method to induce  deformations on $W_k$ with $k>2$ using deformations on $W_2$.
We now formalize this argument.

\begin{proposition}\label{cocyclewk}
Let $k > q > 0$ be 2 positive integers. Then in $\HH^1(W_k, TW_k)$ the class of the section
%\\
\[
\left[
\begin{matrix}
0 \\ z^{-k+q}u_2 \\ 0
\end{matrix}
\right]
\]
is not zero.
\end{proposition}

\begin{proof} This lemma also follows from Thm.\thinspace \ref{Wq} but we present an independent proof
for completeness.
A 1-coboundary with values in $TW_k$ is given by

\begin{align*}
\tau & = \sum_{l \geq 0} \sum_{r \geq 0} \sum_{s \geq 0} 
\left[ \begin{matrix}
\xi^{-2} & 0 & 0 \\
k\xi^{k-1}v_1 & \xi^k & 0 \\
(-k+2)\xi^{-k+1} v_2 & 0 & \xi^{-k+2}
\end{matrix} \right]
\left[ \begin{matrix}
\beta^1_{lrs} \\ \beta^2_{lrs} \\ \beta^3_{lrs}
\end{matrix} \right] 
\xi^l v_1^r v_2^s
+
\left[\begin{matrix}
\alpha^1_{lrs} \\ \alpha^2_{lrs} \\ \alpha^3_{lrs}
\end{matrix} \right] z^l u_1^r u_2^s\\
& = \sum_{l \geq 0} \sum_{r \geq 0} \sum_{s \geq 0} 
\left[\begin{matrix}
\beta^1 z^2 \\
k\beta^1 z^{-1}u_1 + \beta^2 z^{-k} \\
(-k+2)\beta^1 z u_2 + \beta^3 z^{k-2}
\end{matrix} \right]
z^{-l+kr+(-k+2)s}u_1^ru_2^s
+
\left[\begin{matrix}
\alpha^1_{} \\ \alpha^2_{} \\ \alpha^3_{}
\end{matrix} \right] z^l u_1^r u_2^s
\end{align*}
(in the second line we omitted the subindices $lrs$).
Now we see that it is not possible to obtain 
$\left[ \begin{matrix}
0 & z^{-k+q}u_2 & 0
\end{matrix} \right]^T$ as a coboundary. 
Indeed, the possibilities for having $z^{m}u_2$ for $m < 0$ would be  obtained from the 
monomials with  coefficients $\beta^2_{l01}$, for which the power of $z$ is $-2k+2-l$, and thus impossible to be equal to $-k+q$.
\end{proof}

By using the cocycles from Proposition \ref{cocyclewk} we obtain a $(k-1)$-parameter deformation family for $W_k$:

\begin{equation}\label{familywk}
\boxed{
(\xi, v_1, v_2) = \left(z^{-1}, z^ku_1 + \sum_{q=0}^{k-1} t_qz^qu_2, z^{-k+2}u_2 \right).
}
\end{equation}

\begin{theorem}\label{Wq}
Let $k > q > 0$. Then the deformation of $W_k$ given by
\begin{equation}\label{defq}
\boxed{
(\xi, v_1, v_2) = \left(z^{-1}, z^ku_1 + z^qu_2, z^{-k+2}u_2 \right).
}
\end{equation}
is isomorphic to $W_q$.

\end{theorem}

\begin{proof}

Consider the vector bundle over ${\mathbb P^1}$ given by 
\[ \left[
\begin{matrix}
z^k & z^q \\
0 &  z^{-k+2}
\end{matrix}\right].
\]
Changing coordinates we have:
\[
\left[
\begin{matrix}
z^k & z^q \\
0 &  z^{-k+2}
\end{matrix}\right]
\sim
\left[
\begin{matrix}
1 & 0 \\
z^{-k-q+2} & -1
\end{matrix}\right]
\left[
\begin{matrix}
z^k & z^q \\
0 & z^{-k+2}
\end{matrix}\right]
\left[
\begin{matrix}
0 & 1 \\
1 & -z^{k-q}
\end{matrix}\right]
=
\left[
\begin{matrix}
z^q & 0 \\
0 & z^{-q+2}
\end{matrix}\right].
\]
Replacing the extension class with $tz^q$ we also obtain a deformation family 
from $W_k$ to $W_q$.
We may reinterpret this isomorphism as:

\[
\begin{array}{rcl}
\mathcal{W}_k & \to & W_q \\
\left(z, \left[\begin{matrix}
u_1 \\ u_2
\end{matrix} \right] \right) & \mapsto & \left(z, \left[
\begin{matrix}
0 & 1 \\
1 & -z^{k-q}
\end{matrix}\right]^{-1}
\left[\begin{matrix}
u_1 \\ u_2
\end{matrix} \right] \right) \\
\left(\xi, \left[\begin{matrix}
v_1 \\ v_2
\end{matrix} \right] \right) & \mapsto & \left(\xi, 
\left[
\begin{matrix}
1 & 0 \\
\xi^{k+q-2} & -1
\end{matrix}\right]
\left[\begin{matrix}
v_1 \\ v_2
\end{matrix} \right] \right)
\end{array}
\]
where $\mathcal{W}_k$ denotes the deformation of $W_k$ given by Equation \eqref{defq}.
\end{proof}

\begin{corollary}\label{Wkinf}
For $k>1$ the threefold $W_k$ has infinitely many distinct deformations. 
\end{corollary}

\begin{proof} Combine theorems \ref{W2def} and \ref{Wq}.
\end{proof} 

\section{Moduli of vector bundles}

\subsection{Background on moduli spaces}\label{moduli-def}
Moduli space of vector bundles over complex varieties are a classical theme of study in algebraic geometry, with many applications
in various areas of  mathematics and  physics. 
The definition of moduli spaces used in algebraic geometry, though powerful, is rather abstract.  
Given an algebraically closed field $\mathbb{k}$ one requires  a \emph{moduli functor} 
$
\Bun_X^r\colon (\mathfrak{Sch}/\mathbb{k})^{\textnormal{op}} \to \mathfrak{Sets}
$
which associates to a scheme $T \in \mathfrak{Sch}/\mathbb{k}$ the set $\Bun_X^r(T)$ of isomorphism classes of families of rank $r$ vector bundles on $X$ parametrized by $T$.
If the functor $\Bun_X^r$ is representable in the category of schemes, then $\mathcal{M}_X^r$ is called a \emph{fine moduli space} for the moduli problem.
However, in most situations fine moduli spaces do not exist. Instead the notion of  a \emph{coarse moduli space} is used, which is a scheme $\mathcal{M}_X^r$ which \emph{corepresents} $\Bun_X^r$, and  such that $\varphi(\Spec{\mathbb{k}})\colon \Bun_X^r(\Spec{\mathbb{k}}) \to \Hom_{\mathfrak{Sch}/\mathbb{k}}(\Spec{\mathbb{k}}, \mathcal{M}_x^r)$ is a bijection, see for instance \cite{Se}.

Even moduli spaces of well-behaved objects can be ``arbitrarily bad'', i.e. moduli spaces are governed by a kind of ``Murphy's law''. The case of moduli of holomorphic vector bundles on the surfaces $W_k$ seems to be no exception, but we will see that simple choices can provide us with well behaved moduli
spaces that fit under the definition of coarse moduli. We will however continue to work with the analytic topology, not the Zariski one. 

\subsection{The threefolds \texorpdfstring{$W_k$}{W\_k} and their moduli of vector bundles}
We wish to  describe rank 2 vector bundles over $W_k$  in the simplest possible way,
as extensions of line bundles. Hence, the first step is to study filtrability, 
which is a property that holds true for bundles over $W_1$ and $W_2$ but not for the cases when $k\geq 3$. 
We summarize the results of \cite{K} about filtrability.

\begin{lemma}\cite[Thm.~3.10]{K} Every holomorphic vector bundle on $W_1$ is filtrable and algebraic.
\end{lemma}

\begin{proof} Direct application of the more general theorem proved in \cite{BGK1} which uses ampleness
of the conormal bundle of a subvariety, in this case $\mathbb P^1$ inside $W_1$. 
Note that in this case the conormal bundle is $\mathcal{O}_{\mathbb{P}^1}(1) \oplus \mathcal{O}_{\mathbb{P}^1}(1)$, hence ample.

\end{proof}
For $W_2$ there exists  a slightly weaker result, proving filtrability only for algebraic bundles.

We first set the notation.
Let $W_k$ be the total space of $\mathcal{O}{-k}_{\mathbb{P}^1} \oplus \mathcal{O}_{\mathbb{P}^1} (k-2)$ and $\ell \cong \mathbb{P}^1$ the zero section defined by the ideal sheaf $\mathcal{I}_\ell$. 
We write
\[
\ell_N = \left( \ell, \mathcal{O}_{W_k} / \mathcal{I}_\ell^{N+1}\vert_{\ell} \right)
\]
for the $N^{\textnormal{th}}$ neighborhood of $\ell$, $\hat{l} = \lim_{\leftarrow} \ell_N$ for the formal neighborhood of $\ell$ in $W_k$, and $\mathcal{O}(j)$ for the bundle on $W_k$ or on $\hat{\ell}$ that restricts to $\mathcal{O}_{\mathbb{P}^1}(j)$ on $\ell$. 
A vector bundle $E$ has \emph{splitting type} $(j_1, \ldots, j_r)$ if $E\vert\ell \cong \oplus_{i_1}^r \mathcal{O}_{\mathbb{P}^1}(j_i)$ with $j_1 \geq \cdots \geq j_r$.

\begin{theorem}\label{filW2}\cite[Thm.~3.11]{K}
 Let $\ell$ be the zero section of $\mathcal{O}_{\mathbb{P}^1}(-2) \oplus \mathcal{O}_{\mathbb{P}^1}(0)$.
 Fix an integer $r \geq 1$ and a holomorphic rank-$r$ vector bundle $E$ on $\hat{\ell}$.
 Let $a_1 \geq \cdots \geq a_r$ be the splitting type of $E \vert_\ell$. Then there exist vector bundles $F_i$ on $\hat{\ell}$, $0 \leq i \leq r$, such that $F_r \ce E$, $F_1 \ce L_{a_i}$, $F_0 \ce \{0\}$ and $F_i \vert_\ell$ has rank $i$ and splitting type $a_1 \geq \cdots \geq a_i$, and such that there are $r-1$ exact sequences on $\hat{\ell}$ (for $2 \leq i \leq r$)
 \begin{equation}
 0 \longrightarrow L_{a_i} \longrightarrow F_i \longrightarrow F_{i-1} \longrightarrow 0
 \end{equation}
where $L_{a_i} \cong \mathcal{O}(a_i)$.
\end{theorem}

\begin{proof}
 The result is obvious if $r=1$. 
 Hence we may assume $r \geq 2$ and that the result is true for all vector bundles with rank at most $r-1$. 
 By assumption there is an injective map $j\colon \mathcal{O}_\ell(a_r) \to E\vert_\ell$ on $\ell$ such that $\coker(j)$ is a rank-$(r-1)$ vector bundle on $\ell$ with splitting type $a_1 \geq \cdots \geq a_{r-1}$. 
 The map $j$ gives a nowhere-zero section $s$ of $E(-a_r)\vert_\ell$. Let us show that this section extends over a neighbourhood of $\ell$: 
 There is an exact sequence
 \begin{equation} \label{exact.seq}
 0 \longrightarrow S^t(N^*_{\ell,W_2}) \longrightarrow \mathcal{O}^{(t+1)}_\ell \longrightarrow \mathcal{O}^{(t)}_\ell \longrightarrow 0 
 \end{equation}  
 where $S^t(N^*_{\ell,W_2})$ is the $t^{\textnormal{th}}$ symmetric power of the conormal sheaf of $\ell$ in $W_2$. 
 In this case, we have $N^*_{\ell,W_2} \cong \mathcal{O}_\ell(2) \oplus \mathcal{O}_\ell$, therefore,
 \[
 S^t(N^*_{\ell,W_2}) \cong \bigoplus_{k=0}^t \mathcal{O}_\ell(2k).
 \]
 After tensoring by the bundle $E(-a_r)$, the exact sequence \eqref{exact.seq} becomes
 \[
  0 \longrightarrow 
  E(-a_r) \otimes  \left(\bigoplus_{k=0}^t \mathcal{O}_\ell(2k) \right) 
  \longrightarrow 
  E(-a_r) \otimes \mathcal{O}^{(t+1)}_\ell 
  \longrightarrow 
  E(-a_r) \otimes \mathcal{O}^{(t)}_\ell 
  \longrightarrow 0,
 \]
 thus inducing the long exact cohomology sequence
 \[
\cdots \to
\HH^0(\ell, E(-a_r) \otimes \mathcal{O}^{(t+1)}_\ell)
\to
\HH^0(\ell, E(-a_r) \otimes \mathcal{O}^{(t)}_\ell )
\to
\bigoplus_{k=0}^t \HH^1\left(\ell, E(-a_r + 2k)\right)
\to \cdots 
 \]
 Note that $\HH^0(\ell, E(-a_r) \otimes \mathcal{O}^{(t)}_\ell )$ is the space of global sections of $E(-a_r)$ on the $t^{\textnormal{th}}$ infinitesimal neighbourhood of $\ell$ in $W_2$; 
 moreover, the obstruction to extending a section from the $t^{\textnormal{th}}$ infinitesimal neighbourhood to the $(t+1)^{\textnormal{st}}$ one lives in 
 \[
\bigoplus_{k=0}^t \HH^1\left(\ell, E(-a_r + 2k)\right). 
 \]
 However, since $E(-a_r)$ is a bundle of degree $\sum_{i=1}^{r-1}(a_i-a_r) \geq 0$, 
 the splitting type of  
 $E(-a_r+2k)$ consists only of non-negative integers for $0 \leq k \leq t$, and thus all the cohomology groups $\HH^1(\ell, E(-a_r + 2k))$ vanish for $0 \leq k \leq t$.
 Thus any section of $E(-a_r)$ on the $t^{\textnormal{th}}$ infinitesimal neighbourhood extends to the $(t+1)^{\textnormal{st}}$.
 Hence, by Grothendieck's existence theorem \cite[5.1.4]{grothendieck1961elements}, the section $s$ extends to an actual neighbourhood of $\ell$ in $W_2$, and consequently there is an exact sequence on $W_2$ of the form
 \[
0 \longrightarrow
L_{a_r}
\longrightarrow
E
\longrightarrow
F_{r-1}
\longrightarrow 0. 
 \]
 
\end{proof}

\begin{corollary} Every \emph{algebraic} vector bundle on $W_2$ is filtrable.
\end{corollary}

\begin{question} Prove (or disprove) 
filtrability for all holomorphic bundles on $W_2$.
\end{question} 

 An application of the exponential sheaf sequence shows that line bundles on $W_k$ are classified by their first Chern Class.
 We denote by $\mathcal{O}_{W_k}(j)$ the line bundle on $W_k$ with first Chern class $j$. 
 Thus, $\mathcal O_{W_k}(j)=\pi^*\mathcal O_{\mathbb{P}^1}(j)$, where $\pi \colon W_k \rightarrow \mathbb P^1$ is the projection. 
 When it is clear from the context we omit the subscript, and use only $\mathcal O(j)$ instead of $\mathcal O_{W_k}(j)$.

We focus our attention on rank 2 vector bundles with vanishing first Chern class on $W_k$.
If $E$ is such a bundle, then by 
 Grothendieck's splitting principle 
$E\vert_{\mathbb P^1} = \mathcal{O}_{\mathbb P^1}(j) \oplus  \mathcal{O}_{\mathbb P^1}(-j)$
for some integer $j \geq 0$,  which we call the {\it splitting type} of $E$.
We
 can then naively define a "space" of isomorphism classes  of rank 2
vector bundles over $W_k$  of splitting type $j$ as a set by
considering the quotient
\begin{equation} \label{modulidef}
M_j(W_k) = \Ext_{W_k}^1(\mathcal{O}(j),\mathcal{O}(-j)) / \sim
\end{equation}
where  $\sim$ denotes bundle isomorphism.

Such a quotient is rather badly behaved, as expected. Nevertheless, there is a
simple way to extract a moduli space out of it that is a quasi-projective variety. 
This can be done by restricting ourselves to the subset  $\mathfrak M_j(W_k)$ of bundles 
with splitting type $j$ that are defined over 
the first formal neighborhood of the $\mathbb P^1$ inside $W_k$ (extended trivially to higher 
neighborhoods). In other words, in $\mathfrak M_j(W_k)$ only extension classes that vanish to order 
exactly one on $\mathbb P^1$ (and contain no higher order terms on $u_1u_2$) are considered.

We formalize this notation:

\begin{notation}
We denote by $\mathfrak M_j(W_k)$ the subset of $  M_j(W_k)$ consisting of those bundles 
with splitting type $j$  defined by  extension classes of order exactly 1 over $\mathbb P^1$.
 \end{notation}

 K\"oppe described  $\mathfrak M_j(W_k)$ for $k=1,2,3$:

\begin{lemma}\cite[Prop.\thinspace 3.24]{K} \label{dim-split}
The subset of\, $\mathfrak{M}(W_i;j)$, $i=1, 2, 3$, which consists of extensions that do not split on the first infinitesimal neighbourhood $\ell^{(1)}$ is a projective space of dimension 
%$\gamma_1-1$, i.e.,
\[
\dim\left( \mathfrak{M}(W_i;j)\vert_1 \right) = 4j-5 \qquad \textnormal{for } j \geq 2.
\]
Moreover, $\mathfrak{M}(W_1;1)$ is empty, $\mathfrak{M}(W_2;1)\vert_1$ is a point, and $\mathfrak{M}(W_3;1)\vert_1$ is one-dimensional.
\end{lemma}

\begin{proof}
Let $\mathbb{C}^{\gamma_1}$ be the space of coefficients $p_{10s}$ and $p_{01s}$ of $p$, these correspond to coefficients of nontrivial 
extension classes in the first formal neighborhood $\ell^{(1)}$.
Direct verification shows that the only isomorphisms on $\ell^{(1)}$ are scaling, see \cite[Prop.\thinspace3.5]{K}, and thus $\mathfrak{M}(W_i;j)\vert_1$ is obtained by projectivising the open subset of generic coefficients of the affine space $\mathbb{C}^{\gamma_1}$.

Now we just compute $\gamma_1$ directly 
: On the 0-th neighborhood we have $E^0 \ce E\vert_\ell \cong \mathcal{O}_{\mathbb{P}^1}(-j) \oplus \mathcal{O}_{\mathbb{P}^1}(j)$, so that 
\[
\EEnd E^0 \cong \mathcal{O}_{\mathbb{P}^1}(-2j) \oplus \mathcal{O}_{\mathbb{P}^1}(2j) \oplus \mathcal{O}_{\mathbb{P}^1} \oplus \mathcal{O}_{\mathbb{P}^1} \cong
(\EEnd E^0)^ \vee.
\]
Also, the normal bundle of $\ell$ is
\[
N_{\ell, W_1} = \mathcal{O}_{\mathbb{P}^1}(-1)^{\oplus 2} 
\qquad  
N_{\ell, W_2} = \mathcal{O}_{\mathbb{P}^1}(-2) \oplus \mathcal{O}_{\mathbb{P}^1}
\qquad
N_{\ell, W_2} = \mathcal{O}_{\mathbb{P}^1}(-3) \oplus \mathcal{O}_{\mathbb{P}^1}(1),
\]
and by Serre duality,
\begin{align*}
\gamma_1 & = h^0\left( \ell; (\EEnd E^0 \otimes N^*_{\ell, W_i})^\vee \otimes \omega_{\mathbb{P}^1} \right) \\
& = h^0\left( \ell; (\EEnd E^0)^\vee \otimes N_{\ell, W_i} \otimes \omega_{\mathbb{P}^1}  \right) \\
& = 4(j-1) \,\, \textnormal{for $i=1,2,3$ and $j \geq 2$.}
\end{align*}
The results for $j=1$ follow from a similar  computation.
\end{proof}

Since the general extension class is nontrivial on the first neighborhood, we refer to the open subset 
$\mathfrak M_j(W_k)$ of the moduli space as its generic part. So, we may rephrase K\"oppe's result as:

\begin{theorem}\label{dimensionmoduli}
For $k=1,2,3$, the generic part  of the moduli of {\it algebraic}  bundles  $\mathfrak M_j(W_k)$
is smooth of dimension $4j-5$.
\end{theorem}

\begin{proof} In canonical coordinates elements of $\mathfrak M_j(W_k)$ are given by transition matrices 
 \[
\left( \begin{matrix}
z^j & p \\
0 & z^{-j}
\end{matrix} \right), \quad \mbox{where}\quad  p = p_1(z,z^{-1})u_1+p_2(z,z^{-1})u_2,
\]
and isomorphism is given  by projectivization. Calculating $\Ext^1_{W_k}(\mathcal{O}(j),\mathcal{O}(-j))$
 we show that the expression of $p$ has 
$4j-4$ coefficients, and projectivization lowers the dimension by $1$.

\end{proof}
  
 \begin{question}Describe  moduli spaces of  algebraic bundles on $W_k$ for $k>3$. \end{question}
 
 Using canonical coordinate charts for the Calabi--Yau  threefolds 
 $W_k = \Tot(\mathcal O_{\mathbb P^1}(-k))\oplus \mathcal O_{\mathbb P^1}(k-2))$ 
  as in \ref{canonical}, we can represent an element $E_p$  of $\mathfrak M_j(W_k)$ by a transition matrix:
  \[
\left( \begin{matrix}
z^j & p \\
0 & z^{-j}
\end{matrix} \right), \quad \mbox{where}\quad  p \in \Ext^1(\mathcal{O}(j),\mathcal{O}(-j)),
\]
with $p=p(z,z^{-1},u)$. 
Those bundles belonging to the generic part are then the ones having $p$ homogeneous of degree 1 on $u_1,u_2$ 
(because $u_1=u_2=0$ cut out the $\mathbb P^1$ inside $W_k$ on the $U$-chart).

Then, by upper semicontinuity every element near $E_p$ can also be 
represented by an element of $\Ext^1(\mathcal{O}(j),\mathcal{O}(-j))$.
That is, there is a small disk $D$  around $p$ over which the family of bundles near $E_p$ may be represented by 
\[
\left( \begin{matrix}
z^j & p_\alpha \\
0 & z^{-j}
\end{matrix} \right), \quad \mbox{with} \quad p_\alpha \in \Ext^1(\mathcal{O}(j),\mathcal{O}(-j)).
\]
Therefore, the existence of the maps: 
\begin{align*}
D   \rightarrow & \Hom(D, \mathcal{M}_X^r)\\
\alpha \mapsto & (j, \alpha)\equiv  
\left( \begin{matrix}
z^j & p_\alpha \\
0 & z^{-j}
\end{matrix} \right)
\end{align*}
implies that our naively defined quotients satisfy the definition of coarse moduli space if regarded from the point of view of algebraic geometry. 

  The behaviour of moduli changes quite a bit if we consider all holomorphic bundles
  at once (instead of only algebraic).  

Observe that the threefolds $W_k$ admit strictly holomorphic vector bundles which are not algebraic
whenever $k\geq 2$, see for example 
\cite[Cor.\thinspace 5.4]{GKRS}. This contrasts with the result of $Z_k$, where all holomorphic vector bundles are algebraic  \cite{gasparim1}.

  We have:

\begin{theorem}  $W_2$ has infinite-dimensional moduli spaces of holomorphic bundles.
\label{W_2modid}\end{theorem}

\begin{proof} 
\noindent{\it Rank 2.} Consider in general, all  holomorphic bundles  
that correspond to elements of  $\Ext_{W_2}^1(\mathcal{O}(j),\mathcal{O}(-j)) / \sim$. Then the result follows from 
the fact that $$\dim \Ext_{W_2}^1(\mathcal{O}(j),\mathcal{O}(-j)) = h^1(W_2; \mathcal O(-2j)= \infty,$$ by
choosing  different monomials on $u_1,u_2$ of order $n$, these define bundles that split to neighborhood $n-1$ 
but do not split on neighborhood $n$.

\noindent{\it Rank 3.} For brevity we give just an example. Consider the moduli space 
that contains the tangent bundle of $W_2$. The Zariski tangent space of this moduli space  at $TW_2$ 
is given by 
the cohomology $\HH^1(W_2, \End(TW_2))$, which  is infinite-dimensional. Indeed, 
\v{C}ech cohomology calculations show that $\HH^1(W_2, \End(TW_2))$ is 
generated as a $\mathbb C$-vector space by the following cocycles:
\[ (0, \dots, 0, \underbrace{z^{-1}u_1u_2^s}_4, 0 \dots, 0 ), (0, \dots, 0, \underbrace{z^{-i}u_2^s}_4, 0 \dots, 0) \text{ \ for $i = 1,2,3$, and} \]
\[ (0, \dots, 0, \underbrace{z^{-1}u_2^s}_6, 0 \dots, 0 ), (0, \dots, 0, \underbrace{z^{-1}u_2^s}_7, 0 \dots, 0 ) \text{ \ for $s \geq 0$.} \]
\end{proof}

 \begin{question} Describe   moduli spaces of  holomorphic bundles on $W_k$ for $k>1$. \end{question}

\subsection{Moduli of bundles on deformations of $W_2$}

\begin{notation}
Fix $\tau= \sum t_s u_2^s \in \mathcal{O}({\mathbb{C}})$. 
 Then by $\mathcal{W}_2(\tau)$ we mean the deformation of $W_2$ given by 

\[
\boxed{
(\xi, v_1, v_2) = (z^{-1}, z^2u_1 + z\tau, u_2).
}
\]
\end{notation}

\begin{lemma}
$\Pic( \mathcal{W}_2(\tau)) = \mathbb{Z}.$
\end{lemma}

\begin{proof}
It follows from the exponential sheaf sequence and the facts that $ \HH^1(\mathcal{W}_2(\tau), \mathcal{O})= \HH^2(\mathcal{W}_2(\tau), \mathcal{O}) = 0$ and $ \HH^2(\mathcal{W}_2(\tau), \mathbb{Z}) = \mathbb{Z}$.
\end{proof}

Assume $\tau(0) = 0$, then by 
 the moduli $\mathfrak{M}_j(\mathcal{W}_2(\tau))$ of  rank 2 bundles over the surface $\mathcal{W}_2(\tau)$ with splitting type $j$ we mean:
\[
\mathfrak{M}_j(\mathcal{W}_2(\tau)) = \Ext^1_{\mathcal{W}_2(\tau)}(\mathcal{O}(j),\mathcal{O}(-j))/\sim,
\]
where $\sim$ denotes bundle isomorphism, i.e., $p \sim q$ if the vector bundles defined by 
	\[
	\left[ \begin{matrix}
	z^j & p \\ 0 & z^{-j}
	\end{matrix} \right] 
	\textnormal{ and }
	\left[ \begin{matrix}
	z^j & q \\ 0 & z^{-j}
	\end{matrix} \right] 
	\]
are isomorphic, and recall that (to obtain a moduli space) we are considering
only extension classes of first order, that is of the form $p=  p_1(z,z^{-1})u_1+p_2(z,z^{-1})u_2$,
otherwise considering all orders we would have obtained moduli stacks in the sense of \cite{BBG}.

\begin{remark} \label{remark} For special values of $\tau$ we have:
\begin{itemize}
\item If there exists $u \in \mathbb{C}$ such that $\tau(u) =0$, then $\mathcal{W}_2(\tau)$ contains a $\mathbb{P}^1$.
The inclusion is given by:
\[
\begin{array}{rcl}
f_u\colon \mathbb{P}^1 & \hookrightarrow & \mathcal{W}_2(\tau) \\
z & \mapsto & (z,0,u) \\
\xi & \mapsto & (\xi,0,u) 
\end{array}
\]

\item For $\tau \equiv 1$, we have that $\mathcal{W}_2(\tau) = \mathcal{Z}_2 \times \mathbb{C}$, where $\mathcal{Z}_2$ is the nontrivial
(affine)  deformation of $Z_2$.
\end{itemize}

\end{remark}

\noindent {\bf Assumption}: in what follows we consider only those deformations $\mathcal{W}_2(\tau)$ that contain a $\mathbb P^1$ as 
a subvariety (algebraically or analytically). In particular, such deformations are not affine varieties, allowing for the possibility 
of positive dimensional extension groups.

\begin{lemma} \label{h1w2oj}
Elements of  $H^1(\mathcal{W}_2(\tau), \mathcal{O}(-j))$ can be expressed in the form
%%
%\[
%\sum_{l = 1}^{j-i-1} \sum_{i=0}^{j-1} \sum_{s \geq 0 } \sigma_{lis} z^{-l} u_1^i u_2^s.
%\]
%
$$\sum_{s\geq 0 } \sum_{i=0}^{j-1} \sum_{l = 1}^{j-i-1} \sigma_{lis} z^{-l} u_1^i u_2^s.$$
\end{lemma}

\begin{proof}
A 1-cocycle can be expressed in the form
\[
\sigma = \sum_{l = -\infty}^\infty \sum_{i \geq 0 } \sum_{s \geq 0 }  
\sigma_{lis} z^l u_1^i u_2^s .
\] 
 Then
\[
\sigma \sim \sum_{l \leq -1} \sum_{i \geq 0 } \sum_{s \geq 0 }  
\sigma_{lis} z^l u_1^i u_2^s 
\] 
because monomials with positive powers of $z$ are coboundaries.
%
%\[
%\sigma = \sum_{l \geq 0} \sum_{i \geq 0 } \sum_{s \geq 0 }  
%\sigma_{lis} z^l u_1^i u_2^s .
%\] 
%
 Changing coordinates we have $v_2=u_2$ and $v_1= z^2u_1+z\tau$, hence 
 $u_1= \xi^2v_1-\xi\tau$, so the expression for $\sigma$ transforms to 
\begin{align*}
T\sigma & = \sum_{l \leq -1} \sum_{i \geq 0 } \sum_{s\geq 0 }  
\sigma_{lis} z^{l+j} u_1^i u_2^s \\
& = \sum_{l \leq -1} \sum_{i \geq 0 } \sum_{s \geq 0 }  
\sigma_{lis} \xi^{-l-j} (\xi^2v_1- \xi\tau)^i v_2^s \\
& = \sum_{l \leq -1} \sum_{i \geq 0 } \sum_{s \geq 0 }  
\sigma_{lis} \xi^{-l-j+i} (\xi v_1- \tau)^i v_2^s .
\end{align*}
To improve  notation we change $l$ to $-l$ so that we have
\begin{align} \label{qq}
T\sigma & = \sum_{l \geq 1} \sum_{i \geq 0 } \sum_{s \geq 0 }  
\sigma_{lis} \xi^{l-j+i} (\xi v_1- \tau)^i v_2^s.
\end{align}
The holomorphic terms on the  $V$ chart 
are those in which the power of $\xi$ is greater than or equal to 0, so we can remove those terms. The ones that remain satisfy the condition
\[
l + i < j
\]
for $l \geq 1$ and $i \geq 0$. Therefore, we can rewrite  (\ref{qq}) as 
\begin{align*} 
T\sigma & = \sum_{s \geq 0 } \sum_{i=0}^{j-1} \sum_{l = 1}^{j-i-1}    
\sigma_{lis} \xi^{l-j+i} (\xi v_1- \tau)^i v_2^s.
\end{align*}
By changing back to the $U$ chart we have
\[
\sigma \sim \sum_{s\geq 0 } \sum_{i=0}^{j-1} \sum_{l = 1}^{j-i-1} \sigma_{lis} z^{-l} u_1^i u_2^s.
\]
\end{proof}

\begin{lemma}\cite[III.6.3.(c), III.6.7]{Har}  \label{ext-h1}
There is an isomorphism between 
$\Ext^1_{\mathcal{W}_2(\tau)}(\mathcal{O}(j_2),\mathcal{O}(j_1))$
and 
$\HH^1(\mathcal{W}_2(\tau), \mathcal{O}(j_1-j_2))$
given by:

\[
\begin{array}{rcl}
\Ext^1_{\mathcal{W}_2(\tau)}(\mathcal{O}(j_2),\mathcal{O}(j_1)) & \to & \HH^1(\mathcal{W}_2(\tau), \mathcal{O}(j_1-j_2)) \\
\left(\begin{matrix}
z^{-j_1} & p \\ 0 & z^{-j_2}
\end{matrix}\right) & \mapsto & z^{-j_2}p.
\end{array}.
\]
In particular, there is an isomorphism 
$$
\begin{array}{rcl}
\Ext^1_{\mathcal{W}_2(\tau)}(\mathcal{O}(j),\mathcal{O}(-j)) & \to & \HH^1(\mathcal{W}_2(\tau),\mathcal{O}(-2j)) \\
\left(\begin{matrix}
z^{j} & p \\ 0 & z^{-j}
\end{matrix}\right) & \mapsto & z^{-j}p.
\end{array}.
$$
\end{lemma}

\begin{corollary}\label{cor-ext}
Elements of  $\,\Ext^1_{\mathcal{W}_2(\tau)}(\mathcal{O}(j),\mathcal{O}(-j))$ can be expressed in the form
\[
\sum_{s \geq 0 } \sum_{i=0}^{2j-1} \sum_{l = 1}^{2j-i-1} \sigma_{lis} z^{-l+j} u_1^i u_2^s.
\]
\end{corollary}

\begin{proof} It follows from Lemma \ref{h1w2oj} applied to $\mathcal O(-2j)$ and Lemma \ref{ext-h1}.

\end{proof}

\begin{proposition}
Let $\mathcal{W}_2(\tau)$ be a deformation of $W_2$ that contains a $\mathbb{P}^1$ as a subvariety.
 Then every algebraic vector bundle on $\mathcal{W}_2(\tau)$ is filtrable. 
\end{proposition}

\begin{proof}
The transition matrix for the cotangent bundle of $\mathcal{W}_2(\tau)$ is given by:
%\\
\[
\left[\begin{matrix}
-z^{2} & 0 & 0 \\
2zu_1 + \tau(u_2) & z^{-2} & -z^{-1}\tau'(u_2) \\
0 & 0 & 1
\end{matrix} \right]
\]
We then have  the conormal bundle $N_{\mathbb{P}^1, \mathcal{W}_2(\tau)}^\ast = \mathcal{O}_{\mathbb{P}^1}(2) \oplus \mathcal{O}_{\mathbb{P}^1}$,
and we use the same proof as Theorem  \ref{filW2}. 
\end{proof}

\begin{example}\label{group}

An application of Corollary  \ref{cor-ext}  shows that 
 $\Ext^1_{\mathcal{W}_2(\tau)}(\mathcal{O}(2),\mathcal{O}(-2))$ is generated over $\mathbb C$ by monomials of the form
$zu_2^s$, $u_2^s$, $z^{-1}u_2^s$, $zu_1u_2^s$, $u_1u_2^s$,  $zu_1^2u_2^s$ for $s \geq 0$.
So, on the first formal neighborhood we have, in principle, for any $\tau$, generators 
$zu_2$, $u_2$, $z^{-1}u_2$, $zu_1$, $u_1$, some of which may still be zero.
In fact, we observe that $u_1$ will be zero in the moduli space for any choice of $\tau$ because we have the equality
\[
\left[
\begin{matrix}
z^2 & 0 \\ 0 & z^{-2}
\end{matrix}
\right]
=
\left[
\begin{matrix}
1 & 0 \\ 0 & z^{-2}
\end{matrix}
\right]
\left[
\begin{matrix}
z^2 & u_1 \\ 0 & z^{-2}
\end{matrix}
\right]
\left[
\begin{matrix}
1 & -u_1 \\ 0 & z^2
\end{matrix}
\right]
\sim
\left[
\begin{matrix}
z^2 & u_1 \\ 0 & z^{-2}
\end{matrix}
\right],
\]
which is an isomorphism of vector bundles, 
therefore $u_1$ defines the split bundle.
We conclude that the first formal neighborhood of  $\Ext^1_{\mathcal{W}_2(\tau)}(\mathcal{O}(2),\mathcal{O}(-2))$ is generated over $\mathbb C$
 by the monomials $zu_2$, $u_2$, $z^{-1}u_2$, $zu_1$, some of which might still be zero, 
 depending on the choice of $\tau$.
\end{example}

\begin{notation}
We  denote by $E_1$ and  $E_2$ the vector bundles  over $\mathcal{W}_2$ defined respectively by the following transition matrices:
$$
E_1 = 
\left[\begin{matrix}
z^2 & zu_1 \\
0 & z^{-2}
\end{matrix}\right], \quad
E_2 = 
\left[\begin{matrix}
z^2 & zu_2 \\
0 & z^{-2}
\end{matrix}\right].
$$
\end{notation}

\begin{lemma}\label{noniso}
The vector bundles $E_1$ and  $E_2$ over $\mathcal{W}_2$ are not isomorphic.

\end{lemma}

\begin{proof}
 We calculate the generators for $\HH^0(\ell^1, E_i\vert_{\ell^1})$ for $i=1, 2$ as  $\mathbb C$-vector spaces.
 
On the 0-th formal neighborhood, that is over $\ell^0 = \mathbb P^1$ 
 the global sections of  $ E_i\vert_{\ell^0}$ for both $i=1, 2$ are those pulled back from $\mathcal O(2)$, namely 
 $$
s_0=\left[\begin{matrix}
0\\
1
\end{matrix}\right],
s_1=\left[\begin{matrix}
0\\
z
\end{matrix}\right],
s_2=\left[\begin{matrix}
0\\
z^2
\end{matrix}\right].
 $$
We now verify which of these extend to the first formal neighborhood. 
For $n=0,1,2$ let $\tilde s_n=\left[\begin{matrix}
a\\
z^n
\end{matrix}\right]$ denote a possible  extension of $s_n$ to a section of $E_i\vert_{\ell^1}$. 
\vspace{2mm}

\noindent{\sc Case of $E_1$:} We claim that $s_0$  does not extend to $\ell^1$. 
For such an extension $\tilde s_0$  to exist, we must  
have a function $a$ holomorphic on $U= \mathbb C[[z,u_1,u_2]]$ such that 
$$
\left[\begin{matrix}
z^2 & zu_1 \\
0 & z^{-2}
\end{matrix}\right]
\left[\begin{matrix}
a\\
1
\end{matrix}\right]=\left[\begin{matrix}
z^2a+zu_1\\
z^{-2}
\end{matrix}\right]$$
is holomorphic on $V= \mathbb C[[\xi,v_1,v_2]] = \mathbb C[[z^{-1},z^2u_1+zu_2,u_2]]$.
But $zu_1= \xi v_1-v_2$ is itself holomorphic on $V$, therefore we must have $z^2a$ holomorphic on $V$, which 
is impossible on neighborhood 1. So, $s_0$ does not extend to a section of $E_1\vert_{\ell^1}$. 
Direct verification then shows  that $\HH^0(\ell^1, E_1)$ is generated by
$$
\tilde s_1=\left[\begin{matrix}
-u_1\\
z
\end{matrix}\right], \quad
\tilde s_2=\left[\begin{matrix}
-zu_1\\
z^2
\end{matrix}\right].
$$

\noindent{\sc Case of $E_2$:} Direct verification shows that
$ E_2\vert_{\ell1}$ admits the sections:
$$
\tilde s_0=\left[\begin{matrix}
u_1\\
1
\end{matrix}\right],\quad
\tilde s_1= \left[\begin{matrix}
-u_2\\
z
\end{matrix}\right],\quad
\tilde s_2=\left[\begin{matrix}
-u_2\\
z^2
\end{matrix}\right].
$$
We conclude that $h^0(\ell^1, E_1\vert_{\ell^1})=5$ but   $h^0(\ell^1, E_2\vert_{\ell^1})=6$, hence $E_1\nsim E_2$.
\end{proof}

\begin{theorem}\label{decrease}
Some deformations of $W_2$ have the effect of decreasing the dimension of $\mathfrak{M}_2(W_2)$ while keeping the moduli nontrivial.
\end{theorem}

\begin{proof}

In this proof we will denote by $\mathcal{W}_2$ the deformation given by $\tau(u_2) = u_2$.
We will show that the deformed space $\mathfrak{M}_2(\mathcal{W}_2)$ is strictly smaller than the 
original $ \mathfrak{M}_2(W_2)$.

We will use the isomorphism
\[
\left[
\begin{matrix}
z^j & p \\ 0 & z^{-j}
\end{matrix}
\right]
\sim
\left[
\begin{matrix}
1 & \beta \\ 0 & 1
\end{matrix}
\right]
\left[
\begin{matrix}
z^j & p \\ 0 & z^{-j}
\end{matrix}
\right]
\left[
\begin{matrix}
1 & b \\ 0 & 1
\end{matrix}
\right]
=
\left[
\begin{matrix}
z^j & p + z^jb + z^{-j}\beta \\ 0 & z^{-j}
\end{matrix}
\right],
\]
where $b$ is a holomorphic function on $U$ and $\beta$ is a holomorphic function on $V$. 
It implies that the cocycles $p$ and $p + z^jb + z^{-j}\beta$ define isomorphic vector bundles.

The space $\Ext^1_{{W}_2}(\mathcal{O}(2), \mathcal{O}(-2))$ is generated by the non-zero cocycles $z u_1$, $z^{-1}u_2$, $u_2$, $zu_2$.
The space $\Ext^1_{\mathcal{W}_2(\tau)}(\mathcal{O}(2), \mathcal{O}(-2))$ of the deformation $\mathcal{W}_2(\tau)$ is 
generated by the same cocycles, but in these case we will find out that there are further relations among them.

In the case of $\mathcal{W}_2$, by choosing $b=0$ and $\beta=-v_1=-z^2 u_1-zu_2$ we see that 
$u_1$ and $z^{-1}u_2$ define isomorphic bundles, i.e.,
\[
\left[
\begin{matrix}
z^2 & z^{-1}u_2 \\ 0 & z^{-2}
\end{matrix}
\right]
\sim
\left[
\begin{matrix}
z^2 & u_1 \\ 0 & z^{-2}
\end{matrix}
\right]
\sim 
\left[
\begin{matrix}
z^2 & 0 \\ 0 & z^{-2}
\end{matrix}
\right]
,
\]
by Example \ref{group}. 
We conclude that the dimension of $\mathfrak{M}_2(\mathcal{W}_2(\tau)) \leq 2$, hence strictly
smaller than the dimension of $\mathfrak{M}_2(W_2)$, which is 3, by Lemma \ref{dim-split}.

To prove that  the moduli $\mathfrak{M}_2(\mathcal{W}_2)$
is nontrivial, we use Lemma~\ref{noniso} which shows that the classes $z u_1$, and $zu_2$, define nonisomorphic bundles.
\end{proof}

Combining Theorems \ref{decrease} and \ref{Wq} we obtain:

\begin{corollary}\label{fedmoduli} 
For $k\geq 2$ the threefolds $W_k$ admit infinitely many isomorphism classes of  deformations.
Some of these deformations support nontrivial moduli of vector bundles. 
\end{corollary}

\section*{Acknowledgements}
 \noindent
 We are thankful to the referee whose comments and suggestions helped us to improve the paper.
 We completed this paper during a gathering of the Research in Pairs Group 
 on Deformation Theory supported by  the Centro Internazionale per la Ricerca Matematica, of
Fondazione Bruno Kessler, in Povo, Trento. We are very thankful to CIRM 
for the hospitality and wonderful working conditions. Our special thanks to the director Marco Andreatta. 
 
 Ballico was  partially supported by MIUR and GNSAGA of INdAM (Italy).
 Gasparim was partially supported by  a Simons Associateship of the Abdus Salam International Centre for Theoretical Physics, Italy and 
 the Vicerrector\'ia de Investigaci\'on y Desarrollo Tecnol\'ogico de la  Universidad Cat\'olica del Norte (Chile).
 Suzuki acknowledges support from  Beca Doctorado Nacional CONICYT -- Folio 21160257.
 Gasparim and Suzuki express thanks for the  support of  the Network NT08 of the Office of External Activities, ICTP.

%\section*{References}
\bibliography{Wkbib}{}
\bibliographystyle{alpha}

\end{document}